\documentclass[preprint,12pt]{elsarticle}

\usepackage{amsmath,amssymb,amsthm}
\usepackage{mathtools}
\usepackage{xcolor,tikz}
\usepackage[hidelinks]{hyperref}

\theoremstyle{plain}
\newtheorem{theorem}{Theorem}[section]
\newtheorem{proposition}[theorem]{Proposition}
\newtheorem{lemma}[theorem]{Lemma}

\theoremstyle{definition}

\theoremstyle{remark}
\newtheorem{remark}[theorem]{Remark}

\newcommand{\normone}[1]{\lVert #1\rVert_1}

\journal{Discrete Applied Mathematics}

\begin{document}

\begin{frontmatter}

\title{The minimum length of an axis-aligned rectangular tiling of a flat torus}

\cortext[corr]{Corresponding author}
\author{Hau-Yi Lin\corref{corr}\fnref{nycu}}
\ead{zaq1bgt5cde3mju7@gmail.com}
\author{Wu-Hsiung Lin\fnref{nycu}}
\ead{wuhsiunglin@nctu.edu.tw}
\author{Gerard Jennhwa Chang\fnref{ntu}}
\ead{gjchang@math.ntu.edu.tw}
\address[nycu]{Department of Applied Mathematics, National Yang Ming Chiao Tung University, Hsinchu 30010, Taiwan}
\address[ntu]{Department of Mathematics, National Taiwan University, Taipei 10617, Taiwan}

\begin{abstract}
A flat torus is the quotient of the Euclidean plane over a lattice generated by a basis,
and an axis-aligned rectangular tiling of a flat torus
is a partition into finitely many rectangles whose sides are axis-aligned.
We provide the minimum sum of the perimeter of rectangles
for an axis-aligned rectangular tiling,
and prove that it is attainable by either exactly one rectangle or exactly two rectangles.
\end{abstract}


\end{frontmatter}

%%%%%%%%%%%%%%%%%%%%%%%%%%%%%%%%%%%%%%%
\section{Introduction}
%%%%%%%%%%%%%%%%%%%%%%%%%%%%%%%%%%%%%%%

A {\em flat torus} $\mathbb{R}^2/\Lambda$ is a representation of a torus
by the quotient of the Euclidean plane $\mathbb{R}^2$ over a {\em lattice}
$\Lambda$ generated by a basis $\{u,v\}$.
A {\em tiling} of a surface is a family of closed sets with disjoint interiors
whose union is the surface, and is {\em rectangular} if the closed sets are rectangles.
For a rectangular tiling of a flat torus,
we restrict these rectangles to be {\em axis-aligned},
that is, all rectangle sides are aligned with the coordinate axes of the Euclidean plane.

Rectangulation (or rectangular subdivision)
is a classical topic in discrete geometry and geometric graph theory,
related to VLSI floorplanning and rectangular duals of planar graphs
\cite{AckermanBarequetPinter2006,BhaskerSahni1988,KozminskiKinnen1985},
to orthogonal graph drawing \cite{Tamassia1987},
and to feasibility/realizability questions under geometric constraints
such as prescribed areas or aspect ratios \cite{EppsteinMumfordSpeckmannVerbeek2012}.
Recent surveys and structural results highlight the richness of rectangulations
as combinatorial properties \cite{AsinowskiCardinalFelsnerFusy2025,MerinoMutze2023}.
Another study is the optimization for orthogonal partitions,
notably minimizing the total edge length of a rectangular (or rectilinear) partition of an
orthogonal polygon; see
\cite{ChengDuKimRuan2004,DuZhang1990,
GonzalezZheng1990,GonzalezZheng1989,GonzalezZhengHandbook2007,
Levcopoulos1986,LingasPinterRivestShamir1982}.
Toroidal visibility representations and related
rectangular structures have been investigated
from a topological/graph-theoretic viewpoint
\cite{Biedl2022,MoharRosenstiehl1998} and extensions to other
surfaces such as the cylinder also exist \cite{HasheminezhadHashemiMcKayTahmasbi2010}.

In this paper, we consider an optimization problem
that for a flat torus (corresponding to a given basis of $\mathbb{R}^2$),
what is the minimum sum of the perimeter of axis-aligned rectangles
which partition the flat torus.
Equivalently,
since each side of a rectangle is counted twice,
we are asking the minimum length of line segments
for which partition the torus into axis-aligned rectangles,
where the length is half of the sum of the perimeter of rectangles.
Note that the length is dependent on the choice of a basis.
We give a closed-form formula for the minimum length of an axis-aligned rectangular tiling of a flat torus,
and we construct such a tiling with either exactly one rectangle or exactly two rectangles to achieve the minimum.

%%%%%%%%%%%%%%%%%%%%%%%%%%%%%%%%%%%%%
\section{Preliminary}
%%%%%%%%%%%%%%%%%%%%%%%%%%%%%%%%%%%%%

For a basis $\{u,v\}$ of $\mathbb{R}^2$,
let $\Lambda=\mathbb{Z} u+\mathbb{Z} v$ be the lattice generated by $\{u,v\}$,
which is called a {\em $\mathbb{Z}$-basis} of $\Lambda$.
The {\em covolume} of $\Lambda$, denoted by $d(\Lambda)$,
is the area of the parallelogram spanned by $\{u,v\}$,
which is invariant for any basis generating $\Lambda$.

For a flat torus $\mathbb{T}^2=\mathbb{R}^2/\Lambda$,
let $\pi:\mathbb{R}^2\to \mathbb{T}^2$ be the quotient mapping.
An axis-aligned rectangular tiling $T$ of $\mathbb{T}^2$
consists of rectangles $R_1,R_2,\ldots,R_k \subseteq \mathbb{R}^2$,
where $R_i=[a_i,b_i]\times[c_i,d_i]$ for $i=1,2,\ldots,k$,
such that the following conditions hold.
\begin{enumerate}
\item[(i)] $\pi$ is injective on the interior $(a_i,b_i)\times(c_i,d_i)$ of $R_i$ for $i=1,2,\ldots,k$.
\item[(ii)] The images on the interiors of $R_i$'s are pairwise disjoint.
\item[(iii)] $\mathbb{T}^2=\bigcup_{i=1}^k \pi (R_i)$.
\end{enumerate}

The {\em skeleton} of such a tiling $T$ is the graph $G(T)$
whose vertex set consists of the images of corners of all rectangles,
and the images of sides of all rectangles are subdivided by vertices into edges.
The {\em length} $m(T)$ of $T$ is the sum of lengths of all edges of $G(T)$.
Since each rectangle $R_i$ contributes $2[(b_i-a_i)+(d_i-c_i)]$ to the length,
and each edge of $G(T)$ either is incident to exactly two rectangles
or belongs to the opposite sides of one rectangle,
we have that $m(T)$ is half of the sum of perimeters of all $R_i$'s,
that is,
$$
     m(T)=\sum_{i=1}^k[(b_i-a_i)+(d_i-c_i)].
$$

The goal of this article is to find the minimum length $m(T)$,
where $T$ runs over all axis-aligned rectangular tilings of
the flat torus $\mathbb{R}^2/\Lambda$ corresponding to
a lattice $\Lambda$ generated by a given basis of $\mathbb{R}^2$.
For convenience, let $m(\Lambda)=\min\{m(T): T$
is an axis-aligned rectangular tiling of the flat torus $\mathbb{R}^2/\Lambda\}$.

For $u=(x,y)\in\mathbb{R}^2$,
the $\ell^1$-norm of $u$ is denoted by $\normone{u}=|x|+|y|$. Let
$$
    Q_1=\{(x,y):xy\ge0\} \quad\mbox{and}\quad Q_2=\{(x,y):xy<0\}
$$
to be the closed/open ``quadrants'' for the sign of the product.

\begin{theorem}[Minkowski's convex body theorem \cite{Cassels1959}]\label{thm:Minkowski}
Let $\Lambda$ be a lattice.
If $C$ is a convex centrally symmetric region in $\mathbb{R}^2$ with
volume greater than $4 d(\Lambda)$,
then $C$ contains a nonzero lattice point.
\end{theorem}

%quadrant basis

%-------------------------------------------------------
\begin{lemma}\label{lem:qb}
For a lattice $\Lambda$ generated by a basis of $\mathbb{R}^2$, if
$u$ is of minimum $\ell^1$-norm in $(\Lambda\cap Q_1)\setminus\{0\}$ and
$v$ is of minimum $\ell^1$-norm in $\Lambda\cap Q_2$,
then $\{u,v\}$ is a $\mathbb{Z}$-basis of $\Lambda$.
\end{lemma}

\begin{proof}
Write $u = (p,q) \in Q_1$ and $v = (r,s) \in Q_2$.
Replacing $u$ by $-u$ and $v$ by $-v$ if necessary,
we may assume that $p,q \ge 0$ but not both zero as $u \ne 0$, and $r<0<s$.
Let $\Lambda'=\mathbb{Z} u+\mathbb{Z} v\subseteq\Lambda$ be the lattice generated by $\{u,v\}$.
If $\Lambda'\ne \Lambda$,
then $(ps-qr)=d(\Lambda')\ge2 d(\Lambda)$ since $d(\Lambda)\mid d(\Lambda')$.

Consider the convex set symmetric with respect to the origin:
$$
   K=\{(x,y)\in \mathbb{R}^2:
   |x+y|<\normone{u},\,|x-y|<\normone{v}\}.
$$
The area of $K$ equals to $2\normone{u}\normone{v}=2(p+q)(-r+s)> 2(ps-qr)\ge 4d(\Lambda)$,
since $qs-pr>0$.
By Theorem \ref{thm:Minkowski},
the set $K$ contains a nonzero lattice point $w\in\Lambda\setminus\{0\}$
with either $w\in Q_1$ implying $\normone{w}<\normone{u}$
or else $w\in Q_2$ implying $\normone{w}<\normone{v}$,
a contradiction to the minimality of $u$ or $v$.
Therefore, $\Lambda=\Lambda'$ and so $\{u,v\}$ generates $\Lambda$.
\end{proof}

The $\mathbb{Z}$-basis $\{u,v\}$ chosen in Lemma \ref{lem:qb}
is called a {\em quadrant basis} of $\Lambda$.

%%%%%%%%%%%%%%%%%%%%%%%%%%%%%%%%%%%%%%%%%%%%%%%%%%%%%%%%%%
\section{Axis-aligned rectangular tiling of a flat torus}
%%%%%%%%%%%%%%%%%%%%%%%%%%%%%%%%%%%%%%%%%%%%%%%%%%%%%%%%%%
%====================================================
\subsection{Exactly one rectangle}
%====================================================

For a flat torus, we give a sufficient condition for which
there exists an axis-aligned rectangular tiling with exactly one rectangle.

%Horizontal period

%-------------------------------------------------------------
\begin{lemma}\label{lem:hperiod}
Let $\Lambda$ be a lattice of covolume $d(\Lambda)$.
If $(a,0)\in \Lambda$ for some $a\ne 0$,
then for all $(x,y)\in\Lambda$, we have $\frac{d(\Lambda)}{a}\mid y$.
\end{lemma}

\begin{proof}
We may assume that $y\ne 0$
and let $\Lambda'$ be the lattice generated by $(a,0)$ and $(x,y)$.
Since $d(\Lambda)\mid d(\Lambda')$ and $d(\Lambda')=|ay|$, we have
$\frac{d(\Lambda)}{a}\mid y$.
\end{proof}
%\begin{proof}
%We may assume that $y\ne 0$
%and let $\Lambda'$ be the lattice generated by $(d_x,0)$ and $(x,y)$.
%Since $d(\Lambda)\mid d(\Lambda')$ and $d(\Lambda')=d_xy$, we have
%$\frac{d(\Lambda)}{d_x}\mid y$.
%\end{proof}

We set
$$
    m_x(\Lambda):= d_x+\frac{d(\Lambda)}{d_x}
$$
if there exists $(a,0)\in\Lambda $ with $a\ne 0$,
where $d_x$ is the least positive number such that $(d_x,0)\in\Lambda$,
and $m_x(\Lambda):=\infty $ otherwise.
Similarly, we set
$$
    m_y(\Lambda):= d_y+\frac{d(\Lambda)}{d_y}
$$
if there exists $(0,b)\in\Lambda $ with $b\ne 0$,
where $d_y$ is the least positive number such that $(0,d_y)\in\Lambda$,
and $m_y(\Lambda):=\infty$ otherwise.

%[One-rectangle tiling]
%-------------------------------------------------------
\begin{proposition}\label{prop:onerect}
For a flat torus $\mathbb{T}^2=\mathbb{R}^2/\Lambda$,
if $m_x(\Lambda)<\infty$, then
there exists a tiling $T$ of $\mathbb{T}^2$ with exactly one rectangle
$$
     R=[0,d_x]\times[0,\dfrac{d(\Lambda)}{d_x}],
$$
and length $m(T)=m_x(\Lambda)$.
The analogous statement holds for $m_y(\Lambda)$.
\end{proposition}

\begin{proof}
Let $\pi:\mathbb{R}^2\to\mathbb{T}^2$ be the quotient mapping.
First we prove that $\pi$ is injective on the interior of $R$.
If $\pi(x)=\pi (x')$ for some $x,x'\in (0,d_x)\times(0,\frac{d(\Lambda)}{d_x})$,
then $x-x'\in \Lambda$, say $x-x'=(a,b)$.
Since $b\in (-\frac{d(\Lambda)}{d_x},\frac{d(\Lambda)}{d_x})\cap \frac{d(\Lambda)}{d_x}\mathbb{Z}$, we have $b=0$
and hence $a\in (-d_x,d_x)\cap d_x\mathbb{Z}$,
which gives $a=0$ and forces $x=x'$.

Since %$\pi$ is a local isometry on $R$,
the area of $\pi(R)$ is equal to the area of $R$
with value $d_x\times \frac{d(\Lambda)}{d_x}=d(\Lambda)$, the area of $\mathbb{T}^2$,
which implies $\pi(R)=\mathbb{T}^2$.

Therefore, $T$ is an axis-aligned rectangular tiling of $\mathbb{T}^2$ with exactly one
rectangle $R$ and length $m(T)=(d_x-0)+(\frac{d(\Lambda)}{d_x}-0)=m_x(\Lambda)$.
\end{proof}

%======================================================
\subsection{Exactly two rectangles}
%======================================================

Here we give a sufficient condition for which there exists an
axis-aligned rectangular tiling of a flat torus with exactly two rectangles

%[Two-rectangle tiling]
\begin{proposition}\label{prop:tworect}
If $\Lambda$ is a lattice generated by a basis $\{u,v\}$ of $\mathbb{R}^2$
with $u=(p,q)$ and $v=(r,s)$ satisfying $pq>0$ and $rs<0$,
then the flat torus $\mathbb{T}^2=\mathbb{R}^2/\Lambda$
has a rectangular tiling $T$ with exactly two rectangles
$$
    R_1=[-|r|,|p|-|r|]\times[0,|s|] \text{ and }
    R_2=[|p|-|r|,|p|]\times[0,|q|],
$$
and length $m(T)=\normone{u}+\normone{v}$.
\end{proposition}
\begin{proof}
Replacing $u$ by $-u$ and $v$ by $-v$ if necessary,
we may assume that $p,q>0$ and $r<0<s$.

As shown in Figure \ref{fig1}, consider
the parallelogram $P=OACB$ spanned by $\{u,v\}$ in $\mathbb{R}^2$,
where $\overrightarrow{OA}=u=(p,q)$ and $\overrightarrow{OB}=v=(r,s)$.
Let $D,E,F$ be the projection points of $A,B,C$, respectively, onto the $x$-axis;
and let $G,H$ be the projection points of $A,B$, respectively, onto the line segment $\overline{CF}$.
Choose $R_1=BEFH=[r,p+r]\times [0,s]$ and $R_2=ADFG=[p+r,p]\times [0,q]$.

\begin{figure}[htb]
\centering
\begin{tikzpicture}[scale=1/2]\setlength{\fboxsep}{1pt}
\draw [thick, ->](-5,0)--(4,0) node [right]{$x$};
\draw [thick, ->](0,-1)--(0,7) node [above]{$y$};
\foreach \n/\x/\y in {o/0/0,a/3/5,b/-4/1,c/-1/6,d/3/0,e/-4/0,f/-1/0,g/-1/5,h/-1/1}
\path coordinate (\n) at (\x,\y);
\draw[very thick] (a)--(c)--(b) (d)--(e) (a)--(d) (b)--(e) (c)--(f) (a)--(g) (b)--(h) ;
\foreach \n/\N/\p in {o/O/below right,a/A/right,b/B/left,c/C/above,
d/D/below,e/E/below,f/F/below,g/G/left,h/H/right}
\draw[fill] (\n) circle(.2) node[\p]{$\N$};
\draw[line width=3pt,-latex] (o)--node {\colorbox{white}{$u$}}(a);
\draw[line width=3pt,-latex] (o)--node {\colorbox{white}{$v$}}(b);
\end{tikzpicture}
\qquad
\begin{tikzpicture}[scale=1/2]\setlength{\fboxsep}{1pt}
\draw [thick, ->](-5,0)--(4,0) node [right]{$x$};
\draw [thick, ->](0,-1)--(0,7) node [above]{$y$};
\foreach \n/\x/\y in {o/0/0,b/-4/5,a/2/1,c/-2/6,d/2/0,e/-4/0,f/-2/0,g/-2/1,h/-2/5}
\path coordinate (\n) at (\x,\y);
\draw[very thick] (a)--(c)--(b) (d)--(e) (a)--(d) (b)--(e) (c)--(f) (a)--(g) (b)--(h) ;
\foreach \n/\N/\p in {o/O/below right,a/A/right,b/B/left,c/C/above,
d/D/below,e/E/below,f/F/below,g/G/left,h/H/right}
\draw[fill] (\n) circle(.2) node[\p]{$\N$};
\draw[line width=3pt,-latex] (o)--node {\colorbox{white}{$u$}}(a);
\draw[line width=3pt,-latex] (o)--node {\colorbox{white}{$v$}}(b);
\end{tikzpicture}
\caption{Two possible cases for the rectangular tiling with exactly two rectangles.} \label{fig1}
\end{figure}
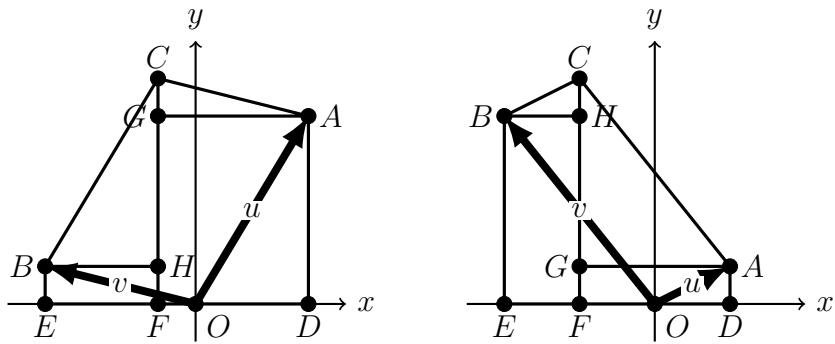

Then the pentagon $ACBED=P\cup OAD\cup OBE= R_1\cup R_2\cup BCH\cup ACG$.
Since $OAD+v=BCH$ and $OBE+u=ACD$, we have that
$T$ is a rectangular tiling of $\mathbb{T}^2$ with rectangles $R_1$ and $R_2$,
and length $m(T)=(p+s)+(-r+q)= (p+q)+(-r+s)=\normone{u}+\normone{v}$.
\end{proof}

%%%%%%%%%%%%%%%%%%%%%%%%%%%%%%%%%%%%%%%%%%%%%%%%%%%%%%%%%%%%%%%%%%%%%%%%%%%%%%%%%
\section{Lower bound on the length of an axis-aligned rectangular tiling}
%%%%%%%%%%%%%%%%%%%%%%%%%%%%%%%%%%%%%%%%%%%%%%%%%%%%%%%%%%%%%%%%%%%%%%%%%%%%%%%%%%

In this section we prove a lower bound on
the length of an axis-aligned rectangular tiling of a flat torus.
Let $G(T)$ and $m(T)$ be the skeleton and the length, respectively, of
an axis-aligned rectangular tiling $T$ of a flat torus corresponding to a given lattice $\Lambda$.

%============================================================================
\subsection{Skeletons that contain a horizontal/vertical cycle}
%============================================================================

%[Horizontal/vertical closed curves]

%--------------------------------------------
\begin{proposition}\label{prop:axiscyle}
If $G(T)$ contains a horizontal cycle, then $m(T)\ge m_x(\Lambda)$.
If $G(T)$ contains a vertical cycle, then $m(T)\ge m_y(\Lambda)$.
\end{proposition}
\begin{proof}
By symmetry, we only show the horizontal case.

Let $\gamma$ be a horizontal cycle of $G(T)$.
We have that $\Lambda$ contains a lattice point $(|\gamma|,0)$
where $|\gamma|$ is the total length of edges in $\gamma$,
and hence $|\gamma|=d_x$ by definition.

Next, we construct a closed walk $W$ by following steps.
\begin{enumerate}
\item[(S1)]
Start from any vertex and stop at a repeated vertex.
\item[(S2)]
Traverse a vertical edge upward to the next vertex until no such edge exists.
\item[(S3)]
Traverse a horizontal edge rightward to the next vertex and go to step (S2).
\end{enumerate}
Because $G(T)$ has finitely many vertices,
there is a closed walk $W$ between the first pair of repeated vertices.
By the construction,
even the walk traverses into a horizontal cycle,
it traverses upward before a repeated vertex.
Thus, $W$ contains at least one vertical edge,
and no edge is repeated in $W$.
By Lemma \ref{lem:hperiod},
the total length of vertical edges in $W$ is a multiple of $\frac{d(\Lambda)}{d_x}$.
Together with $\gamma$,
we have $m(T)\ge|\gamma|+\frac{d(\Lambda)}{d_x}= d_x+\frac{d(\Lambda)}{d_x}=m_x(\Lambda)$.
\end{proof}

%============================================================================
\subsection{Skeletons that contain no horizontal/vertical cycle}
%============================================================================

Here, we focus on the tiling with length less than both $m_x(\Lambda)$ and $m_y(\Lambda)$.
A path in $G(T)$ is a {\em maximal horizontal path}
if it consists of horizontal edges only
and not a sub-path of other horizontal paths.
A {\em maximal vertical path} is defined analogously.

%[Perimeter reduction]

%-------------------------------------------------
\begin{lemma}\label{lem:reduce}
If there exists an axis-aligned rectangular tiling $T$ of the flat torus $\mathbb{T}^2=\mathbb{R}^2/\Lambda$
corresponding to a lattice $\Lambda$ with length $m(T)<\min\{m_x(\Lambda)$, $m_y(\Lambda)\}$,
then there exists another axis-aligned rectangular tiling $T'$ of $\mathbb{T}^2$ with length $m(T')\le m(T)$
and skeleton $G(T')$ has both exactly one maximal horizontal path and
exactly one maximal vertical path.
\end{lemma}

\begin{proof}
Clearly $G(T)$ contains no horizontal/vertical cycle; for otherwise,
$m(T)\ge\min\{m_x(\Lambda),m_y(\Lambda)\}$.
Hence $G(T)$ contains horizontal and vertical paths.
Let $H$ be a maximal horizontal path in $G(T)$.
We classify those rectangles in $T$
with a side contained in $H$ into three types.
\begin{itemize}
\itemsep -3pt
\item $S_1$: rectangles both whose top and bottom sides are mapped onto $H$.
\item $S_2$: rectangles only whose top side is mapped onto $H$.
\item $S_3$: rectangles only whose bottom side is mapped onto $H$.
\end{itemize}

By symmetry, we may assume $|S_2|\ge |S_3|$, and consider two cases.

{\bf Case 1.}
$|S_2|=0$.
In this case, $|S_3|=0$.
If there is a rectangle not in $S_1$,
then $G(T)$ contains a vertical cycle
whose edges are formed from the sides
belonging to both one rectangle in $S_1$ and one not in $S_1$.
By Proposition \ref{prop:axiscyle},
we have $m(T)>m_y(\Lambda)$ which gives a contradiction.
So, $S_1$ consists of all rectangles in $T$
and hence $H$ is the only maximal horizontal path.

{\bf Case 2.}
$|S_2|>0$.
Assume that $T$ consist of rectangles $R_i=[a_i,b_i]\times[c_i,d_i]$ for $i=1,2,\ldots,k$,
and let $h=\min\{d_i-c_i: \pi(R_i)\in S_2\}$,
where $\pi$ is the quotient mapping.
Then we can construct a new tiling $T'$ with rectangles:
$$
  R_i'= \left\{\begin{array}{ll}
   {[a_i,b_i] \times [c_i-h,d_i-h]},
                     &$ if $\pi(R_i)\in S_1; \\
   {[a_i,b_i] \times [c_i,d_i-h]},
                     &$ if $\pi(R_i)\in S_2; \\
   {[a_i,b_i] \times [c_i-h,d_i]},
                     &$ if $\pi(R_i)\in S_3; \\
    R_i              & $ otherwise$.
         \end{array}\right.
$$
Since at least one rectangle $R_j'$ is eliminated for which $\pi(R_j)\in S_2$ and $d_j-c_j=h$,
the bottom side of $R_j$ is not mapped onto $H$,
and so $m(T')\leq m(T)-(b_j-a_j)-h(|S_2|-|S_3|)<m(T)$.
Meanwhile, $G(T')$ has fewer maximal horizontal paths than {\color{red} $G(T)$},
and at most as many maximal vertical paths as $G(T)$.

Since $G(T)$ has finitely many horizontal/vertical paths,
with similar argument for vertical paths,
finally we can obtain a tiling $T'$ with length
$m(T')\le m(T)<\min\{m_x(\Lambda),m_y(\Lambda)\}$
and $G(T')$ has both exactly one maximal horizontal path
and exactly one maximal vertical path.
\end{proof}

%[Cross lower bound]

%----------------------------------------------
\begin{proposition}\label{prop:cross}
If $G(T)$ contains no horizontal/vertical cycle,
and has both exactly one maximal horizontal path
and exactly one maximal vertical path,
then $m(T)\ge \normone{u}+\normone{v}$,
where $\{u,v\}$ is a quadrant basis of $\Lambda$.
\end{proposition}
\begin{proof}
Let $H=\pi(\overline{AB})$ be the maximal horizontal path,
where $\overline{AB}$ is a horizontal line segment in $\mathbb{R}^2$ with
$A$ the left endpoint and $B$ the right endpoint;
and let $V=\pi(\overline{CD})$ be the maximal vertical path,
where $\overline{CD}$ is a vertical line segment in $\mathbb{R}^2$ with
$C$ the bottom endpoint and $D$ the top endpoint.
In $G(T)$, both $\pi(A)$ and $\pi(B)$ are internal vertices of $V$,
and symmetrically both $\pi(C)$ and $\pi(D)$ are internal vertices of $H$.
Let $A',B'\in\overline{CD}$ be the points with $\pi(A')=\pi(A)$ and $\pi(B')=\pi(B)$;
and similarly let
$C',D'\in\overline{AB}$ be the points with $\pi(C')=\pi(C)$ and $\pi(D')=\pi(D)$.

Consider the following four vectors:
\begin{align*}
  w_1 &= \overrightarrow{AC'}+\overrightarrow{CA'}, &
  w_2 &= \overrightarrow{BC'}+\overrightarrow{CB'}, &
  w_3 &= \overrightarrow{BD'}+\overrightarrow{DB'}, &
  w_4 &= \overrightarrow{AD'}+\overrightarrow{DA'}.
\end{align*}
We have $w_i\in\Lambda$ for $i=1,2,3,4$.
Moreover, $w_1,w_3\in Q_1$ and  $w_2,w_4\in Q_2$, which gives
$\normone{w_1}+\normone{w_2}+\normone{w_3}+\normone{w_4}
\ge 2\normone{u}+ 2\normone{v}$.

By the hypotheses, $m(T)$ is the sum of length of edges in $H$ and $V$. Thus
\begin{eqnarray*}
  m(T) &=&   |\overline{AB}|+|\overline{CD}|\\
       &=&   \frac 12 \Big[ (|\overrightarrow{AC'}|+|\overrightarrow{BC'}|)
             +(|\overrightarrow{AD'}|+|\overrightarrow{BD'}|) \\
       & &   +(|\overrightarrow{CA'}|+|\overrightarrow{DA'}|)
             +(|\overrightarrow{CB'}|+|\overrightarrow{DB'}|) \Big] \\
       &=&   \frac 12 \Big[ (|\overrightarrow{AC'}|+|\overrightarrow{CA'}|)
             +(|\overrightarrow{BC'}|+|\overrightarrow{CB'}|) \\
       & &   +(|\overrightarrow{BD'}|+|\overrightarrow{DB'}|)
             +(|\overrightarrow{AD'}|+|\overrightarrow{DA'}|) \Big] \\
       &=&   \frac12 (\normone{w_1}+\normone{w_2}+\normone{w_3}+\normone{w_4}) \\
       &\ge& \frac12 (2\normone{u}+2 \normone{v})
             = \normone{u}+ \normone{v}.
\end{eqnarray*}
We finish the proof.
\end{proof}

%%%%%%%%%%%%%%%%%%%%%%%%%%%%%%%%%%%%%%%%%%%%%%%%%%%%%%%%%%%%%%%%%%%%%%%%%%%%%%%
\section{Minimum length of an axis-aligned rectangular tiling}
%%%%%%%%%%%%%%%%%%%%%%%%%%%%%%%%%%%%%%%%%%%%%%%%%%%%%%%%%%%%%%%%%%%%%%%%%%%%%%%

Now we give a lower bound and prove that it is attainable.

%[Minimum skeleton length]

%-------------------------------------
\begin{theorem}\label{thm:main}
Let $\mathbb{T}^2=\mathbb{R}^2/\Lambda$ be the flat torus
corresponding to a lattice $\Lambda$ generated by a basis of $\mathbb{R}^2$.
The minimum length of an axis-aligned rectangular tiling of $\mathbb{T}^2$ is
$$
     m(\Lambda)=\min\left\{\normone{u}+\normone{v}, m_x(\Lambda), m_y(\Lambda)\right\},
$$
where $\{u,v\}$ is a quadrant basis of $\Lambda$.
\end{theorem}

\begin{proof}
For a tiling $T$, if $m(T)<\min\{m_x(\Lambda),m_y(\Lambda)\}$,
then by Lemma \ref{lem:reduce}, there exists a tiling $T'$
with $m(T')\le m(T)$ and $G(T')$ contains no horizontal/vertical cycle,
and has both exactly one maximal horizontal path
and exactly one maximal vertical path.
So Proposition~\ref{prop:cross} yields $m(T')\ge \normone{u}+\normone{v}$.
Hence
$m(T)\ge \min\left\{\normone{u}+\normone{v}, m_x(\Lambda), m_y(\Lambda)\right\}$.

If $m_x(\Lambda)$ (respectively, $m_y(\Lambda)$) is finite,
then Proposition~\ref{prop:onerect}
provides a tiling $T$ with one rectangle and length $m(T)=m_x(\Lambda)$ (respectively, $m_y(\Lambda)$).
When $\normone{u}+\normone{v}<\min\{m_x(\Lambda),m_y(\Lambda)\}$,
let $u=(p,q)$ and $v=(r,s)$ with $pq\ge0$ and $rs<0$.
If $p=0$ (respectively, $q=0$), then
$|q|=d_y$ (respectively, $|p|=d_x$)
and $|r|=\frac{d(\Lambda)}{d_y}$ (respectively, $|s|=\frac{d(\Lambda)}{d_x}$),
so $\normone{u}+\normone{v}=|p|+|q|+|r|+|s| > m_y(\Lambda)$ (respectively, $m_x(\Lambda)$).
So we have $pq>0$, and
Proposition~\ref{prop:tworect} gives a tiling $T$ with two rectangles
and length $m(T)=\normone{u}+\normone{v}$.

Hence $m(\Lambda)=\min\left\{\normone{u}+\normone{v}, m_x(\Lambda), m_y(\Lambda)\right\}.$
\end{proof}

\section{Computation remarks}

Let $\{u,v\}$ be any $\mathbb{Z}$-basis of $\Lambda$ with $u=(p,q)$ and $v=(r,s)$.

%----------------------------------------------------------------
\begin{remark}[Computing $m_x(\Lambda)$ and $m_y(\Lambda)$]

There exists $(a,0)\in\Lambda$ with $a\ne 0$
(equivalently, $m_x(\Lambda)<\infty$)
if and only if $\frac{q}{s}$ or $\frac{s}{q}$ is rational.
In that case $\frac{d(\Lambda)}{d_x}=\gcd(q,s)$
can be obtained by the Euclidean algorithm, and
$$
    m_x(\Lambda)=d_x+\frac{d(\Lambda)}{d_x}
    =\frac {|ps-qr|}{\gcd(q,s)}+\gcd(q,s).
$$
Analogously $m_y(\Lambda)=\frac {|ps-qr|}{\gcd(p,r)}+\gcd(p,r)$ is obtained.
\end{remark}

%-----------------------------------------------------------------------
\begin{remark}[Finding a quadrant basis of $\Lambda$]

For any lattice $\Lambda\subset\mathbb{R}^2$,
both $(\Lambda\cap Q_1)\setminus\{0\}$ and $\Lambda\cap Q_2$ are nonempty.
For computation, write $\Lambda = B\mathbb{Z}^2$,
where $B = \begin{pmatrix} p & r \\ q & s \end{pmatrix}$.
For any $z\in\mathbb{Z}^2$, set $x=Bz\in\Lambda$.
Then
$$
   \normone{z}
 = \normone{B^{-1}x}  \le \normone{B^{-1}}\normone{x},
   \qquad
   \normone{B^{-1}}=\frac{\max\{|s|+|q|,\ |r|+|p|\}}{|ps-qr|}.
$$
Hence for any radius $r>0$, every lattice point $x\in\Lambda$ with $\normone{x}\le r$
arises from an integer vector $z$ with $\normone{z}\le\normone{B^{-1}} r$,
so there are only finitely many candidates to check.
Therefore one can find a quadrant basis
by enumerating $z\in\mathbb{Z}^2$ in increasing $\normone{z}$ and keeping the best candidate in each $Q_1$ or $Q_2$.
\end{remark}

\end{document}